\documentclass[10pt, a4paper, reqno]{amsart}

\usepackage{amsmath}
\usepackage{amssymb}
\usepackage{amsthm}
\usepackage[all]{xy}
\usepackage{enumerate}
\usepackage{bbm}

\usepackage[utf8]{inputenc}
\usepackage[T1]{fontenc}
\usepackage[russian,english]{babel}

\newcommand{\Z}{\mathbb{Z}}
\newcommand{\Q}{\mathbb{Q}}

\newcommand{\lie}{\mathfrak}
\newcommand{\ds}{\lie{ds}}
\newcommand{\nfz}{\lie{nfz}}
\newcommand{\grt}{\lie{grt}}
\newcommand{\FZ}{\mathcal{F\!Z}}
\DeclareMathOperator{\Ua}{\mathcal{U}}
\newcommand{\Uds}{\Ua\ds}

\DeclareMathOperator{\SL}{SL}
\DeclareMathOperator{\Lie}{Lie}
\newcommand{\Lieb}{\mathbb{L}}

\newcommand{\scal}[2]{( #1 \mid #2 )}
\newcommand{\chev}[1]{\langle #1 \rangle}
\DeclareMathOperator{\ad}{ad}
\DeclareMathOperator{\id}{id}
\DeclareMathOperator{\Ker}{Ker}
\newcommand{\transp}[1]{{}^{t}\!#1}
\DeclareMathOperator{\diag}{diag}

\renewcommand{\le}{\leqslant}
\renewcommand{\ge}{\geqslant}

\newcommand{\sha}{\mathrel{\text{\foreignlanguage{russian}{ш}}}}

\theoremstyle{definition}
\newtheorem*{defi}{Definition}
\newtheorem*{ex}{Example}
\theoremstyle{remark}
\newtheorem*{rem}{Remark}
\theoremstyle{plain}
\newtheorem{thm}{Theorem}[section]
\newtheorem{prop}[thm]{Proposition}
\newtheorem*{claim}{Claim}
\newtheorem*{coro}{Corollary}

\title[Period polynomial relations between double zeta values]{Period polynomial relations\\ between double zeta values}
\author[S. Baumard and L. Schneps]{Samuel Baumard and Leila Schneps}
\date{June 2011}

\allowdisplaybreaks[3]

\begin{document}

\maketitle

\begin{abstract}
 The even weight period polynomial relations in the double shuffle Lie algebra~$\ds$ were discovered by Ihara, and completely classified in~\cite{Schn} by relating them to restricted even period polynomials associated to cusp forms on~$\SL_2(\Z)$. In an article published in the same year, Gangl, Kaneko and~Zagier~\cite{GKZ} displayed certain linear combinations of odd-component double zeta values which are equal to scalar multiples of simple zeta values in even weight, and also related them to restricted even period polynomials. In this paper, we relate the two sets of relations, showing how they can be deduced from each other by duality.
\end{abstract}

\medskip
\section{Introduction and background}\label{section1}

We recall the definition of the even period polynomials associated to cusp forms on~$\SL_2(\Z)$ of even weight~$k$, and the associated space~$E_k$ of the restricted period polynomials. We also recall the isomorphisms between~$E_k$ and~a set of special elements of the double shuffle Lie algebra~$\ds$, and between~$E_k$ and~a set of particular relations between double zeta values in the formal double zeta space. Based on these results, the goal of the paper is to show how to deduce each of these isomorphisms explicitly and directly from the other by duality.

\subsection{Restricted period polynomials associated to cusp forms on~$\SL_2(\Z)$}

\addtocounter{footnote}{1}
\begin{defi} The period polynomials in weight~$k$ are given by
\[ r(X)=\int_{0}^{\infty} f(z)\,(z-X)^{k-2}\,\mathrm{d}z \]
for cusp forms~$f(z)\in S_k(\SL_2(\Z))$. The even period polynomials~$r^-(X)$ are obtained from these by
\[ r^-(X)=\frac{r(X)+r(-X)}{2}\text{.}\]

\medskip\noindent
Finally, the \emph{restricted even period polynomials}~$p(X)$ are obtained from~$r^-(X)$ by subtracting off the term of top degree~$X^{k-2}$ and the constant term. These polynomials, homogenized by an additional variable~$Y$ to degree~$k-2$, form~a vector space which we denote by~$E_k$.
\end{defi}

By the Eichler-Shimura correspondence, the map
\[S_k({\rm SL}_2(\Z))\rightarrow E_k\]
induced by the above definition is an isomorphism. It is easy to deduce from the work of Zagier~\cite{ZagInv,ZagCoh} on period polynomials that~a polynomial~$P(X)$ lies in~$E_k$ if and only if it is even of degree~$\le k-4$, without constant term, and satisfies the relations~$P(X)+X^{k-2}P\bigl(\frac{1}{X}\bigr)=0$ and
\[P(X)+X^{k-2}P\bigl(1-\tfrac{1}{X}\bigr)+(X-1)^{k-2}P\bigl(\tfrac{1}{1-X}\bigr)=0\text{.}\] 

\pagebreak[3]
The first restricted even period polynomials, in weights~$12$ and~$16$, are 

\begin{equation}\label{polysper}
 \begin{cases}\phantom{\text{and}\quad 2\,} (X^8-X^2)-3\,(X^6-X^4)  &   \\
\text{and}\quad 2\,(X^{12}-X^2)-7\,(X^{10}-X^4)+11\,(X^8-X^6)\text{.}  &  
\end{cases}
\end{equation}

\subsection{The double shuffle Lie algebra and period polynomial elements}\label{subsection2}

Let~$\Lie[x,y]$ denote the free Lie algebra on two generators, and for each~$n\ge 1$, let~$\Lie_n[x,y]$ denote the subvector space of homogeneous Lie polynomials of degree~$n$. Let~$\Lie_{\ge n}[x,y]$ denote the Lie algebra of Lie polynomials all of whose monomials are of degree greater than or equal to~$n$, i.e.~$\Lie_{\ge n}[x,y]=\bigoplus_{m\ge n}\Lie_m[x,y]$.

For~a polynomial~$f$ in~$x$ and~$y$ and any word~$w$ in~$x$ and~$y$, let~$\scal{f}{w}$ denote the coefficient of the word~$w$ in~$f$, and extend this notation to~$\scal{f}{g}$ for~$g\in \Q\langle x,y\rangle$ by right linearity.

Let~$\ds$ denote the double shuffle Lie algebra defined as follows:
\[\ds=\{f\in \Lie_{\ge 3}[x,y] \mid  \scal{f}{u*v}=0\}
\]
where~$u*v$ represents the stuffle product of all nontrivial words~$u$ and~$v$ in~$x$ and~$y$ both ending in~$y$\footnote{Let~$y_i=x^{i-1}y$, so that all words ending in~$y$ can be written as~$y_{i_1}\cdots y_{i_r}$. The stuffle product of two such words is defined recursively by~$1*w=w*1=w$ and~$y_iw*y_jw'=y_i\,(w*y_jw')+y_j\,(y_iw*w')+y_{i+j}\,(w*w')$.} and not both powers of~$y$.

For every~$f\in \Q\chev{x,y}$, let~$D_f$ be the associated derivation of~$\Q\chev{x,y}$ defined on the generators by~$D_f(x)=0$ and~$D_g(y)=[y,f]$. The Poisson bracket 
\[\{f,g\}=[f,g]+D_f(g)-D_g(f)\]
arises naturally from bracketing derivations, since as is easily verified, we have
\[[D_f,D_g]=D_f\circ D_g-D_g\circ D_f=D_{\{f,g\}}\text{.}\]

The following result was shown independently and using quite different methods by Goncharov~\cite{GonG} (who displayed an explicit cobracket on the dual space), Racinet~\cite{Rac} (who gave~a direct proof), and \'Ecalle~\cite{Eca} (who embedded the whole situation in~a much vaster theory).

\begin{prop}
The vector space~$\ds$ is~a Lie algebra under the Poisson bracket.
\end{prop}

The Lie algebra~$\ds$ is graded by the \emph{weight}, which is the homogeneous degree of~a Lie polynomial. We write~$\ds_n$ for the weight~$n$ graded part. The full algebra~$\ds$ and each of its graded parts~$\ds_n$ are filtered by the \emph{depth}, which is defined to be the smallest number~$d$ such that an element~$f\in \ds$ contains~a monomial (with non-zero coefficient) having~$d$ letters~$y$. This depth filtration is decreasing on~$\ds$ and on each subspace~$\ds_n$: one has~a decreasing sequence~$\ds_{n}^{1}\supset \ds_{n}^{2}\supset \ds_{n}^{3}\supset\cdots$. 

The depth filtration is not~a grading, due to the existence of linear combinations of depth~$d$ elements whose depth is strictly greater than~$d$. The most studied of these relations, which conjecturally generate the Lie ideal of all such relations, are the \emph{period polynomial relations} which we now proceed to define.

\medskip
Various closely related proofs (essentially based on using the Drinfel'd associator) have been given for the following result (cf. \cite{ZagKyo,Eca,GonG,Rac}).

\begin{prop} For every odd~$n\ge 3$, there exists~a non-zero element~$f\in \ds_n$ of depth~$1$.
\end{prop}

Thanks to this result, we now fix~a choice of depth~$1$ elements~$f_n\in \ds_n$ for each odd~$n\ge 3$ once and for all. The results we prove below are independent of this choice.

\begin{defi} For every even weight~$k$, the \emph{period polynomial} elements in~$\ds_k^2$ are linear combinations~$P$ of Poisson brackets of~$f_n$ and~$f_{k-n}$ such that
\begin{equation}\label{eqpp}
 P=\sum_{\substack{n=3 \\ n\text{ odd}}}^{k-3} a_{n}\,\{f_n,f_{k-n}\}\equiv 0\ \pmod{\ds_k^3}\text{.}
\end{equation}
\end{defi}

The depth~$2$ part of~$P$ comes only from the Poisson brackets of the depth~$1$ parts of the~$f_n$; thus the existence of such relations only depends on these depth~$1$ parts (cf.~\cite{Schn}). Up to scalar multiple, the depth~$1$ part of~$f_n$ is equal to the Lie word~$\ad_{x}^{n-1}(y)$. Thus, the coefficients~$a_n$ appearing in the period polynomial elements are independent of the choice of depth~$1$ elements~$f_n$. In this paper we assume that the~$f_n$ are normalized so that the coefficient of~$\ad_{x}^{n-1}(y)$ is equal to~$1$. 

Working in~a slightly different context (that of the stable derivation Lie algebra, also known as the Grothendieck-Teichm\"uller Lie algebra~$\grt$ and consisting of polynomials with integer coefficients), Ihara~\cite{IhaG} first discovered the existence of elements such as~(\ref{eqpp}), which are zero modulo depth~$3$. He gave the first such element, which occurs in weight~$12$, as 
\[2\,\{f_3,f_9\}-27\,\{f_5,f_7\}\equiv 0 \pmod{\grt_{12}^3}\text{.}\]
Here, he normalized the chosen~$f_n$ by taking them with relatively prime integer coefficients. Ihara also discovered~a similar relation in weight 16, and together with Takao~\cite{IhaG}, proved that the dimension of the space of such relations in even weight~$k$ is given by
\[\bigl[ \tfrac{k-4}{4} \bigr] - \bigl[\tfrac{k-2}{6}\bigr]\text{,}\]

\medskip\noindent
which is exactly the dimension of the space~$S_k(\SL_2(\Z))$ of cusp forms of weight~$k$ on~$\SL_2(\Z)$ (this result was also proven in~\cite{Gond}). As indicated in~\cite{Schn}, these special elements depend only on the depth~$1$ terms~$\ad_{x}^{n-1}(y)$ and are thus entirely independent of the definition of the stable derivation algebra~$\grt$. In fact, the same linear combinations will have the property of being zero modulo depth~$3$ in the Lie algebra~$\Lieb[x,y]$ whose underlying vector space is the same as~$\Lie[x,y]$ but which is equipped with the Poisson bracket, and in any Lie subalgebra of~$\Lieb[x,y]$ having~a depth~$1$ element for each odd~$n\ge 3$, so in both~$\grt$ and~$\ds$. 

If the chosen~$f_n$ are normalized by taking the coefficient of the Lie word~$\ad_{x}^{n-1}(y)$ to be equal to~$1$, then Ihara's relations in weights~$12$ and~$16$ become
\begin{equation}\label{relatsper}\begin{cases}
&\hfill\{f_3,f_9\}-3\,\{f_5,f_7\}\equiv 0\pmod{\ds_{12}^3}\phantom{\text{.}}   \\
\text{and}&2\,\{f_3,f_{13}\}-7\,\{f_5,f_{11}\}+11\,\{f_7,f_9\}\equiv 0\pmod{\ds_{16}^3}\text{.}  
\end{cases}
\end{equation}

The main result of~\cite{Schn} proves the result suggested by comparison of~(\ref{polysper}) and~(\ref{relatsper}), namely:

\begin{thm}[\cite{Schn}]\label{thmschn}Let~$L$ denote any sub-Lie algebra of~$\Lieb$ containing an element~$f_n$ in each weight~$n$ satisfying~$\scal{f_n }{  x^{n-1}y}=1$. Then for each even~$k$, the following conditions are equivalent:
 \begin{enumerate}[(i)]
  \item~$\displaystyle \sum_{i=1}^{[\frac{k-4}{4}]} a_i\,\{f_{2i+1},f_{k-2i-1}\}\equiv 0\pmod{L_{k}^{3}}~$ ;
  \item~$\displaystyle P(X,Y)=\sum_{i=1}^{[\frac{k-4}{4}]} a_i\, (X^{2i}Y^{k-2-2i}-
X^{k-2i-2}Y^{2i})\in E_k~$. 
 \end{enumerate}
\end{thm}

\subsection{Period polynomial relations between double zeta values}

The \emph{multiple zeta values} are defined by the series
\begin{equation}\label{eqmzv}
 \zeta(r_1,\ldots,r_k)=\sum_{n_1>\cdots >n_k>0} \frac{1}{n_1^{r_1}
\cdots n_k^{r_k}}
\end{equation}
where the~$r_i$ are strictly positive integers with~$r_1>1$ to ensure convergence. The \emph{double zeta values} are the values of the convergent series
\[\zeta(r,s)=\sum_{m>n>0} \frac{1}{m^rn^s}\]
whenever~$r$ and~$s$ are positive integers with~$r>1$. In the paper~\cite{GKZ}, Gangl, Kaneko and Zagier exhibit for even~$k\ge 12$ some particular linear combinations of double zeta values~$\zeta(r,s)$ with odd entries~$r$ and~$s$ such that~$r+s=k$, which have the property of being equal to~a scalar multiple of~$\zeta(k)=\sum_{n>0}\frac{1}{n^k}$, and which also arise from period polynomials. 

More precisely, the result of~\cite{GKZ} on period polynomials and odd-component double zetas which concerns us here can be expressed as follows. 

\begin{thm}[\cite{GKZ}]\label{thmgkz}
Let~$k\ge 12$ be even. Let~$P(X,Y)\in E_k$ and set
\begin{equation}\label{changevar}P(X+Y,Y)=\sum_{r=1}^{k-3} \binom{k-2}{r-1}\,q_{r,k-r}\,X^{r-1}Y^{k-r-1}\text{.}\end{equation}
Then 
\begin{equation}\label{eqgkz}
\sum_{\substack{r=3 \\ r\text{ odd}}}^{k-3} q_{r,k-r}\, \zeta(r,k-r)\equiv 0\pmod{\zeta(k)}\text{.}
\end{equation}
Conversely, the weight~$k$ even polynomials~$Q$ giving rise to a relation of the type~(\ref{eqgkz}) are those for which the polynomial~$P$ defined by~(\ref{changevar}) is a period polynomial.
\end{thm}

We observe that this statement does not appear exactly in this form in~\cite{GKZ}; instead, they give~a rough version in theorem 3 of the introduction (``The values~$\zeta(\mathrm{od},\mathrm{od})$ of weight~$k$ satisfy at least~$\dim S_k$ linearly independent relations, where~$S_k$ denotes the space of cusp forms of weight~$k$ on~$\SL_2(\Z)$''), and~a more refined version in the body of the article, in which both even and odd component double zetas are studied in~a ``formal zeta space''~$D_k$. The result of theorem~\ref{thmgkz} is indicated for real multizeta values in remark 2 following Theorem 3 of~\cite{GKZ}. That remark works in exactly the same way space~$\mathcal{FZ}$ of formal multizeta values which we define in the following section (and which is not the same as their space of ``formal double zetas'', even in depth 2). 

The first GKZ relation occurs in weight~$k=12$ and is given by
\begin{equation}\label{relgkz12}
 28\,\zeta(9,3)+150\,\zeta(7,5)+168\,\zeta(5,7)=\frac{5197}{691}\,\zeta(12)\text{.}
\end{equation}

In the remarks ending the introduction of~\cite{GKZ}, several questions are raised concerning the relation between their results and the period polynomial results of~\cite{Schn}, and the relations observed by Ihara in the stable derivation Lie algebra. The present article provides some of these connections; the main result in particular (proposition~\ref{thmiso} and the final corollary) shows how to explicitly deduce the existence and coefficients of the GKZ relations of theorem~\ref{thmgkz} (in both real and formal zetas) from the existence of the period polynomial elements in~$\ds$ and vice versa.

\section{Universal enveloping algebra and duality}\label{sectionUds}

In order to work simultaneously on the level of multiple zeta values and double shuffle, we place ourselves within the following extremely useful diagram which simultaneously shows all the levels and dualities between the Hopf algebras, Lie algebras and Lie coalgebras in which double shuffle and double zeta relations are generally studied.

\smallskip
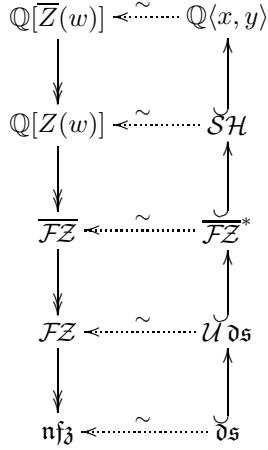
\[\xymatrix{
\Q[\overline Z(w)]\ar@{->>}[d] & \Q\chev{x,y}\ar@{.>}[l]_\sim\\
\Q[Z(w)]\ar@{->>}[d] & \mathcal{SH}\ar@{^{(}->}[u]\ar@{.>}[l]_\sim\\
\overline{\FZ}\ar@{->>}[d] & \overline\FZ^*\ar@{^{(}->}[u]\ar@{.>}[l]_\sim\\
\FZ\ar@{->>}[d] &  \Uds\ar@{^{(}->}[u]\ar@{.>}[l]_\sim\\
\nfz & \ds\ar@{^{(}->}[u]\ar@{.>}[l]_\sim
}\]\label{diagramme}

\bigskip\noindent
The top right-hand space is the underlying vector space of the free polynomial ring on two non-commutative variables~$x$ and~$y$, with basis the set of words in~$x$ and~$y$, graded by \emph{weight} (i.e. degree of monomials). The top left-hand space~$\Q[\overline Z(w)]$ is its graded dual, which is the direct sum of the duals of the graded parts of~$\Q\chev{x,y}$. As~a vector space, we equip the graded dual of~$\Q\chev{x,y}$ with~a dual basis to the basis of words~$w$ in~$x$ and~$y$, and write~$\overline Z(w)$ for the dual basis element associated to~a word~$w$, so that
\[\scal{\overline Z(w)}{v}=\delta^w_v\text{.}\]

\bigskip\noindent
All the vector spaces in the diagram except for the bottom two are actually equipped with Hopf algebra structures. We do not need to specify the explicit multiplication and coproduct on each space, but we do note that~$\Q\langle x,y\rangle$, the free polynomial ring on two non-commutative variables~$x$ and~$y$, is equipped with the standard coproduct given by~$\Delta(x)=x\otimes 1+1\otimes x$ and~$\Delta(y)=y\otimes 1+1\otimes y$, and that its dual~$\Q[\overline Z(w)]$, where~$\overline Z(w)$ denotes the dual basis element of~a word~$w\in \Q\langle x,y\rangle$, is equipped with the multiplication dual to this coproduct, which is the shuffle multiplication of words~$\overline Z(u)\,\overline Z(v)=\overline Z(u\sha v)$\footnote{The shuffle product of words is defined recursively by~$w\sha 1=1\sha w=w$,~${s\,u\sha t\,v} = {s\,(u\sha t\,v )} + {t\,( s\,u \sha v)}$ where~$t,s\in \{x,y\}$.}. The multiplication~$\odot$ on~$\Q\langle x,y\rangle$ is \emph{not} the simple concatenation, but~a complicated rule which we do not know how to write down explicitly in general (but see the beginning of section~\ref{section3} for certain cases). The top horizontal map is the duality isomorphism~$w\mapsto \overline Z(w)$.

Let us define the other spaces and maps in the diagram. 

We say that~a word~$w$ in~$x$ and~$y$ is \emph{convergent} if~$w=xvy$ for any word~$v$ (even~a constant). The second left-hand space down in the diagram,~$\Q[Z(w)]$, is the quotient of the top space by the \emph{shuffle regularization relations} defined for all non-convergent words as follows. Let~$w=y^r\,v\,x^s$ where~$v$ is~a convergent word. Then~$\Q[Z(w)]$ is the quotient of~$\Q[\overline Z(w)]$ by the linear relations
\begin{equation}\label{eqreg}
\overline Z(w)=\sum_{a=0}^r \sum_{b=0}^s (-1)^{a+b}
\overline{Z}\bigl(\pi(y^a\sha y^{r-a}\,u\,x^{s-b} \sha x^b)\bigr)
\end{equation}
where~$\pi$ denotes the projection of polynomials onto just their convergent words and~$Z$ is considered to be linear on sums of words. We write~$Z(w)$ for the image of~$\overline Z(w)$ in this quotient; by definition,~$\Q[Z(w)]$ is spanned by the symbols~$Z(w)$ for convergent~$w$. It is~a long-established theorem that the formula~(\ref{eqreg}) (known as ``shuffle regularization'', and given explicitly by Furusho~\cite{Furu} building on work by Le-Murakami, Zagier and others) ensures that
\[Z(w)\,Z(w')=Z(w\sha w')\]
 for all words~$w,w'$, so that the algebra structure of~$\Q[\overline Z(w)]$ under shuffle multiplication descends to~$\Q[Z(w)]$. Formula~(\ref{eqreg}) shows in particular that~$Z(x)=Z(y)=0$ in~$\Q[Z(w)]$. The dual space is~a subspace of~$\Q\langle x,y\rangle$ denoted by~$\mathcal{SH}$. 

Although the horizontal map is obviously~a duality isomorphism, we have chosen to write it as~a map~$\mathcal{SH}\rightarrow \Q[Z(w)]$ and similarly for all the lower levels of the diagram, simply because the map in this direction is easier to describe: if~$f$ lies in any of the right-hand spaces, then we may consider~$f=\sum a_ww\in \Q\langle x,y \rangle$, and the horizontal duality maps are obtained by taking the image of the dual element~$f^*=\sum a_w\overline Z(w)$ in the corresponding quotient on the left-hand side.

The next space down,~$\overline{\FZ}$, is obtained by quotienting~$\Q[Z(w)]$ by the \emph{regularized stuffle relations}
\begin{equation}\label{eqstu}
Z^*(u)\,Z^*(v)=Z^*(u*v)
\end{equation}
for all words~$u$ and~$v$ ending in~$y$, where~$Z^*$ is defined as follows:
\[\begin{cases}
\hfill Z^*(v)=Z(v) &  \text{for convergent words }v \\
\hfill Z^*(\underbrace{1,\ldots,1}_n) &  \text{is defined by~(\ref{eq10}) below} \\
Z^*(y^mv)=\displaystyle\sum_{r=0}^m Z^*(\underbrace{1,\ldots,1}_r)Z(y^{m-r}v) &  \text{for convergent }v\text{,}
\end{cases}
\]
where we set
\begin{equation}\label{eq10}\exp \sum_{r\ge 1} \tfrac{(-1)^{r-1}}{r} \, Z(x^{r-1}y)\,y^r=
\sum_{r\ge 0} Z^*(\underbrace{1,\ldots,1}_r)\,y^r\text{.}
\end{equation}

Using this definition to transform the~$Z^*$'s into algebraic expressions in the symbols~$Z(w)$ for convergent words, and then using shuffle multiplication to transform each product into~a linear combination of convergent~$Z(w)$, equation~(\ref{eqstu}) translates into a set of linear relations between the convergent symbols~$Z(w)\in \Q[Z(w)]$, and~$\overline{\FZ}$ is obtained by quotienting~$\Q[Z(w)]$ by these linear relations. We continue to write~$Z(w)$ for the image of this element in~$\overline{\FZ}$ and~$\FZ$ by a slight abuse of notation.

For convergent words~$w=x^{r_1-1}y\cdots x^{r_k-1}y$ with~$r_1>1$, let us write 
\begin{equation}\label{eq11}
 Z(w)=Z(r_1,\ldots,r_k)
\end{equation}
in order to underline the equivalence between the~$x,y$-word notation for formal zetas and the usual notation for real zetas. Since it is well-known and easy to prove that real multiple zeta values satisfy the double shuffle relations, the formal zeta space~$\overline{\FZ}$ surjects to the space of real multiple zeta values defined in~(\ref{eqmzv}) simply via 
\[Z(r_1,\ldots,r_k)\mapsto \zeta(r_1,\ldots,r_k)\label{.}\]

The space~$\FZ$ is the quotient of~$\overline{\FZ}$ by the ideal generated by~$Z(xy)=Z(2)$, and the \emph{new zeta space}~$\nfz$ is the vector space obtained by quotienting~${\FZ}$ by the subspace generated by~${\FZ}_0=\Q$ and by all products~${\FZ}_{\ge 1}^2$ (note that every space in the diagram is graded by weight).

The new zeta space~$\nfz$ is~a Lie coalgebra, as was shown by Goncharov~\cite{GonG} who displayed an explicit Lie cobracket on it, dual to the Poisson bracket. Its dual is the Lie algebra~$\ds$ (a complete proof of this elementary fact is not easy to find in the literature, but was given for example in~\cite[prop.~1.27]{Carr}). Thus, the elements of~$\ds$ are the set of primitive elements for the coproduct on~$\FZ^*$ (which is just the restriction of the coproduct~$\Delta$ on~$\Q\langle x,y\rangle$). From the Milnor-Moore theorem (or~a corollary of it, cf.~\cite[theorem~1.22]{CM}), since~$\FZ$ is~a positively graded commutative Hopf algebra over~$\Q$ such that~$\FZ_0=\Q$ and each of its graded pieces is finite-dimensional, its dual~$\FZ^*$ is isomorphic to the universal enveloping algebra of its set of primitive elements; thus~$\FZ^*\simeq \Uds$.

This completes the definition of all the left-hand spaces in the diagram, and, by duality, the right-hand spaces. Since the latter are all vector spaces which are subspaces of the polynomial ring~$\Q\langle x,y\rangle$, the spaces on the right-hand side of the diagram can all be computed explicitly (in small weight) as the spaces killed by the kernels of each quotient map on the left. For example, since~$\scal{Z(w)}{f}=\scal{f}{w}$, the space~$\mathcal{SH}$ consists of polynomials~$f(x,y)$ such that for every non-convergent word~$w=y^rvx^s$, we have
\[\scal{f}{w}-\sum_{a=0}^r \sum_{b=0}^s (-1)^{a+b} 
\scal{f }{  \pi(y^a \sha y^{r-a}vx^{s-b} \sha x^b}=0\text{.}\]
The dimension of the weight~$n$ part~$\mathcal{SH}_n$ is equal to~$2^{n-2}$.

As we saw for~$\ds$ in section~\ref{section1}, the right-hand spaces are all equipped with~a depth filtration defined by taking the depth of~a polynomial~$f$ in~$x$ and~$y$ to be the minimal number of~$y$'s occurring in any monomial of~$f$ with non-zero coefficient.

The vector spaces on the left-hand side of the diagram can then all be equipped with the dual depth filtrations. Namely, for each space~$V$ on the left, we set~$V^d$ to be the subset of~$V$ annihilated by~$(V^*)^{d+1}$, with~$V^0=0$. Thus we have
\begin{equation}\label{eq12}
 (V^*)_n^d/(V^*)_n^{d+1}\simeq V_n^d/V_n^{d-1}\text{.}
\end{equation}
By the action~$\scal{Z(w)}{f}=\scal{f}{w}$, it is obvious that if~$w$ is~a convergent word containing~$d-1$~$y$'s and~$f$ is~a polynomial of depth~$\ge d$, then~$\langle Z(w),f \rangle=0$, so the depth filtration on the spaces on the left corresponds to the usual notion of depth filtration of multiple zeta values, for which the depth of~$Z(w)$ for~a convergent word~$w$ is the number of~$y$'s in~$w$.

\section{The main result}\label{section3}

The main result of this paper says that the GKZ relations~(\ref{eqgkz}), which are valid in the formal multiple zeta algebra~$\overline{\FZ}$, can be deduced in even weight~$k\ge 12$ directly from the period polynomial relations~(\ref{eqpp}) in~$\ds$ by duality, and vice versa.

\begin{thm}\label{mainthm}
 There is a natural bijective correspondence connecting relations of theorem~\ref{thmgkz} between double zeta values and relations of theorem~\ref{thmschn} between Poisson brackets of Lie polynomials.
\end{thm}

The rest of this section is devoted to the proof of this result. The strategy is to show that these relations can be described respectively as lying in the kernel of a certain matrix and of its transpose.

\medskip
In order to prove theorem~\ref{mainthm}, we first rephrase the main result of~\cite{Schn} as follows, showing that in fact it generalizes from~$\ds$ to the universal enveloping algebra~$\Uds$. The notation~$\odot$ indicates multiplication of elements of~$\ds$ in the universal enveloping algebra. The explicit formula for the multiplication of two polynomials~$f$ and~$g$ in~$\Uds$ is complicated, but in the case where~$f$ lies in fact in~$\ds$, it simplifies to
\begin{equation}\label{odot}f\odot g=fg+D_f(g)\end{equation}
where~$D_f$ is the derivation defined in paragraph~\ref{subsection2}.

\begin{prop}\label{thmiso}
Let~$k\ge 12$ be even and fix~a choice of depth~$1$ elements~$f_n$ of~$\ds_n$ for every odd~$n\ge 3$. For~$1\le i\le \frac{k-4}{2}$, set
\[\begin{cases}
w_i=x^{2i}yx^{k-2i-2}y\\
Z_i=Z(2i+1,k-2i-1)=Z(w_i)\in\overline{\FZ}\\
g_i=f_{2i+1}\odot f_{k-2i-1}\in \Q\chev{x,y}\text{.}
\end{cases}\]
Let~$A$ be the~$\frac{k-4}{2}\times \frac{k-4}{2}$ matrix defined by
\[A_{ij}=\scal{Z_i}{g_j}\qquad \text{for}\qquad 1\le j\le \tfrac{k-4}{2}\text{.}\]
Then there is an isomorphism~$E_k\simeq \Ker A$ given by 
\[\sum_{i=1}^{[\frac{k-4}{4}]} a_i\,(X^{2i}Y^{k-2-2i}-X^{k-2-2i}
Y^{2i})\longmapsto (a_1,\ldots,a_{\frac{k-4}{2}})\in \Ker A\text{.}\]
\end{prop}

\begin{proof}
The kernel of~$A$ is the set of vectors such that the corresponding linear combinations of the~$g_i$ are annihilated by the~$Z_i$. In~\cite{GKZ}, it is proved that in even weight~$k$, the odd-component double zeta values~$Z(r,s)$ with~$r,s\equiv 1\pmod{2}$ span the depth~$2$ part~$\overline{\FZ}_k^2$ of~$\overline{\FZ}_k$ (loc. cit., theorem 2, using the fact that~$\overline{\FZ}_k^2$ is~a quotient of~$D_k$). Therefore, since the~$Z_i$ together with~$Z(k)$ span the depth~$2$ part~$\overline{\FZ}^2_k$, and~$Z(k)$ automatically annihilates all the~$g_i$ since they are all of depth~$2$, this means that the kernel of~$A$ is the set of linear combinations
\begin{equation}
\label{eq13}\sum_{i=1}^{\frac{k-4}{2}} a_i\,g_i
\end{equation}
which are zero modulo~$\overline{\FZ}_k^3$. It follows from~\cite{Schn} that there is an injective map~$E_k\hookrightarrow \Ker A$; by the property~$P(X)=X^{k-2}P(\frac1X)$ satisfied by the elements of~$E_k$, we have 
\[a_i=-a_{\frac{k-2i-2}2}\]

\medskip\noindent
for~$1\le i\le \frac{k-4}{2}$, so that all the linear relations between the~$g_i$ arising from polynomials in~$E_k$ are actually linear relations amongst the pairs~$g_i-g_{(k-2i-2)/2}=f_{2i+1}\odot f_{k-2i-1}-f_{k-2i-1}\odot f_{2i+1}=\{f_{2i+1},f_{k-2i-1}\}$. Since it is shown in~\cite{Schn} that there are no other linear relations between these brackets, to prove the proposition we need only show that there can be no linear combination~(\ref{eq13}) between the~$g_i$ which is zero modulo depth 3 but has~$a_i\ne -a_{(k-2i-2)/2}$ for some~$i$.

To do this, it is convenient to compute the elements of the matrix explicitly, which is not difficult given the expression~(\ref{odot}) of the~$\odot$ multiplication. As the~$Z_{i}$'s are of depth~$2$, the desired scalar product is in fact equal to
\begin{align*}
 A_{ij} &= \scal{Z(x^{2i}yx^{k-2-2i}y)}{f_{2j+1}\odot f_{k-2j-1}} \\ &= \scal{Z(x^{2i}yx^{k-2-2i}y)}{\ad_{x}^{2j}(y)\,\ad_{x}^{k-2-2j}(y) + D_{\ad_{x}^{2j}(y)}(\ad_{x}^{k-2-2j}(y))} \\
&= \scal{Z(x^{2i}yx^{k-2-2i}y)}{\ad_{x}^{2j}(y)\ad_{x}^{k-2-2j}(y) + \ad_{x}^{k-2-2j}([y,\ad_{x}^{k-2-2j}(y)])}
\end{align*}
using the definitions of~$D$ and~$\odot$. Then, explicitly computing the coefficients of monomials in the Lie brackets, we find that
\[A_{ij}=\binom{2j}{2i}-\binom{2j}{k-2-2i}+\delta_{i+j}^{\frac{k-2}2}\text{.}\]
Let~$S$ be the~$\frac{k-4}{2}\times \frac{k-4}{2}$ matrix with~$-1$'s along the antidiagonal, so~$S^2=\id$. Let us make the base change to~a basis of eigenvectors of~$S$:
\[\begin{cases}
v_j = \transp(\underbrace{0,\ldots,0}_{j-1},1,0,\ldots,0,1,\underbrace{0,\ldots
,0}_{j-1})  &  \text{for }1\le j\le \left(\frac{k-4}{4}\right)\text{, with eigenvalue }-1 \\
w_0 = \transp(\underbrace{0,\ldots,0}_{(k-6)/4},1,\underbrace{0,\ldots,0}_{(k-6)/4})  & 
\text{if }k\equiv 2\pmod{4} \\
w_j = \transp(\underbrace{0,\ldots,0}_{j-1},-1,0,\ldots,0,1,\underbrace{0,\ldots
,0}_{j-1}) & \text{for }1\le j\le \bigl[\frac{k-4}{4}\bigr]\text{, with eigenvalue }1\text{.}
\end{cases}
\]

Let~$T$ be the matrix having columns~$v_1,\ldots,v_{[\frac{k-4}{4}]},w_0,w_1,\ldots,w_{[\frac{k-4}{4}]}$ in that order (with the~$w_0$ left out if~$k\equiv 0\pmod{4}$). To show that every vector in the kernel of~$A$ is of the form~$\transp(a_1,a_2,\ldots,-a_2,-a_1)$, we use~$T$ to make the basis change from the standard basis to the~$v_i$ and~$w_j$, and then show the equivalent result that the kernel of the matrix~$M=T^{-1}AT$ lies in the space generated by the~$w_j$.

But this result is an immediate consequence of the following claim on the form of~$M$, so it remains only to prove this claim.

\begin{claim}
The matrix~$M=T^{-1}AT$ is~a block matrix of the form 
\begin{equation}\label{eq14}
 M=T^{-1}AT=\begin{pmatrix}\id & 0 \\  B & C\end{pmatrix}\text{,}
\end{equation}
where all four blocks are of dimension~$\frac{k-4}4\times \frac{k-4}4$ if~$k\equiv 0\pmod{4}$, whereas if~$k\equiv 2\pmod{4}$, the identity block is of dimension~$\frac{k-2}4\times \frac{k-2}4$ and the~$0$ block of dimension~$\frac{k-2}4\times \frac{k-6}4$. 
\end{claim}

\medskip
\noindent\emph{Proof of claim.} 
The calculation of the matrix entries turns out to be particularly easy since all the binomial coefficients of the~$A_{ij}$ cancel out and it is merely~a matter of checking the Kronecker deltas. We drop the upper index of the deltas since it is always equal to~$\frac{k-2}2$.

Suppose first that~$k\equiv 0 \pmod 4$. For the upper left-hand block~$1\le i,j\le \frac{k-4}{4}$, we have
\[M_{ij}=A_{ij}+A_{i,\frac{k-2}{2}-j}+A_{\frac{k-2}{2}-i,j}
+A_{\frac{k-2}{2}-i,\frac{k-2}{2}-j}\text{,}\] 
so
\[M_{ii}=\tfrac{1}{2}(\delta_{2i}+2\delta_{\frac{k-2}{2}}+\delta_{k-2i-2})=1\]
and
\[M_{ij}=\tfrac{1}{2}(\delta_{i+j}+\delta_{\frac{k-2}{2}+i-j}\delta_{\frac{k-2}{2}-i+j}+\delta_{k-2-i-j})=0\ \ {\rm if}\ \ i\ne j\text{,}\]
so this block is indeed just the identity matrix.

For the upper right-hand block~$1\le i\le \frac{k-4}{4}$ and~$\frac{k-2}{4}\le j\le \frac{k-4}{2}$, we have
\begin{align*}
 M_{ij}  & = -A_{ij}+A_{i,\frac{k-2}{2}-j}-A_{\frac{k-2}{2}-i,j}+A_{\frac{k-2}{2}-i,\frac{k-2}{2}-j}\\
  & = \tfrac{1}{2}(-\delta_{i+j}+\delta_{\frac{k-2}{2}+i-j}-\delta_{\frac{k-2}{2}-i+j}+\delta_{k-2-i-j})=0\text{,}
\end{align*}
since if~$i+j=\frac{k-2}{2}$ the first and last deltas cancel, whereas if~$i=j$ then the middle deltas cancel. So the upper right-hand block is a~$0$ block, completing the proof in the case~$k\equiv 0\pmod 4$.

This calculation above remains valid in the case~$k\equiv 2\pmod4$, but it is not complete; it shows that the upper left~$\frac{k-6}{4}\times \frac{k-6}{4}$ block is the identity, whereas the upper right~$\frac{k-6}{4}\times \frac{k-2}{4}$ block is zero, but we still have to determine the~$(\frac{k-2}{4})$-th row. We have
\[M_{\frac{k-2}{4},j}=
\begin{cases}
A_{\frac{k-2}{4},j}+A_{\frac{k-2}{4},
\frac{k-2}{2}-j} & \text{if }1\le j\le \frac{k-6}{4} \\
A_{\frac{k-2}{4},\frac{k-2}{4}} &  \text{if }j=\frac{k-2}{4} \\
-A_{\frac{k-2}{4},j}+A_{\frac{k-2}{4}, \frac{k-2}{2}-j} & 
\text{if }\frac{k+2}{4}\le j\le \frac{k-4}{2}\text{.}
\end{cases}
\]
Again, the binomial coefficients cancel trivially, so this becomes
\[
M_{\frac{k-2}{4},j}=
\begin{cases}
 \delta_{\frac{k-2}{4}+j}+
\delta_{\frac{3k-6}{4}-j} &  \text{if }1\le j\le \frac{k-6}{4} \\
\delta_{\frac{k-2}{2}} &  \text{if }j=\frac{k-2}{4} \\
-\delta_{\frac{k-2}{4}+j}+\delta_{\frac{3k-6}{4}-j} &  \text{if }\frac{k+2}{4}\le j\le \frac{k-4}{2}\text{,}
\end{cases}
\]
all of which are zero except when~$j=\frac{k-2}{4}$, in which case the value is equal to~$1$, completing the proof of the claim, and thus of the proposition. 
\end{proof}

This result shows that there are no other linear relations modulo depth~$3$ between the elements~$f_{2i+1}\odot f_{k-2i-1}$ than the already known period polynomial relations between Poisson brackets~$\{f_{2i+1},f_{k-2i-1}\}$, and thus that the kernel of~$A$ consists exactly in vectors whose coefficients are the coefficients of period polynomials~$P\in E_k$.

\begin{ex}
In weight~$k=12$, the matrices~$A$ and~$T^{-1}AT$ are given by
\[A=\begin{pmatrix}1 & 6 & 15 & 28\\ 0 & 1 & 15 & 42\\ 0 & 0 & -14 & -42\\ 0 & -6 & -15 & -27\end{pmatrix}\text{,}
\qquad 
T^{-1}AT=\begin{pmatrix}1 & 0 & 0 & 0\\ 0 & 1 & 0 & 0\\ -28 & -21 & -27 & -9\\ -42 & -15 & -42 & -14\end{pmatrix}\cdot
\]
The kernel of~$A$ is generated by the weight~$12$ period polynomial vector~$\transp(1,-3,3,-1)$ corresponding to the only linear relation modulo depth~$3$ between~$f_3\odot f_9$, $f_5\odot f_7$, $f_7\odot f_5$ and~$f_9\odot f_3$, namely
\[f_3\odot f_9-3\,f_5\odot f_7+3\,f_7\odot f_5-f_9\odot f_3=\{f_3,f_9\}-3\,\{f_5,f_7\}\equiv 0
\pmod{\text{depth }3}\text{.}\]
\end{ex}

From this example, we can already perceive how the GKZ relations between odd-component double zetas arise in this situation. Indeed, as we saw, the kernel of~$A$ is the set of linear relations between generators of~$(\overline{\FZ}^*)_{12}^2 / (\overline{\FZ}^*)_{12}^3$. Thus by duality, the kernel of the transpose~$\transp A$ is the set of linear relations between the odd-component double zetas~$Z(3,9)$,~$Z(5,7)$,~$Z(7,5)$,~$Z(9,3)$ in the dual space~$\FZ_{12}^2/\FZ_{12}^1$, which correspond to linear combinations of the odd-component double zetas which are equal to~a scalar multiple of~$Z(k)$. One computes explicitly that the kernel of~$\transp A$ is generated by the vector~$\transp(0,168,150,28)$; therefore we know without further investigation that 
\[168\,Z(5,7)+150\,Z(7,5)+28\,Z(9,3)\equiv 0\pmod{Z(12)}\text{,}\]
thus recovering relation~(\ref{relgkz12}) except for the coefficient of~$Z(12)$ (the coefficients of the~$Z(k)$ terms in general even weight~$k$ are computed in~\cite{GKZ}).

This argument generalizes to the following statement, which is really the heart of the deduction of the double zeta relations from the period polynomial relations in the double shuffle Lie algebra and vice versa.

\begin{coro}
For all even~$k\ge 16$, the space~$E_k$ of weight~$k$ restricted period polynomials is in bijection with the kernel of the matrix~$A$, itself in bijection with the set of linear relations between the Poisson brackets~$\{f_{2i+1},f_{k-2i-1}\}$. Thus the kernel of the transpose matrix~$\transp A$ has the same dimension as~$E_k$, and is in bijection with the set of~$\Q$-linear relations between the odd-component weight~$k$ double zetas~$Z(r,s)$ and~$Z(k)$.
\end{coro}

\begin{proof}
Since the vectors of~$\Ker A$ correspond to linear relations between the~$g_i$ viewed as generators of~$(\overline{\FZ}^*)_k^2/ (\overline{\FZ}^*)_k^3$, the vectors of~$\Ker\transp A$ correspond to linear relations between the~$Z_i$ (odd-component double zetas) viewed as generators of the dual space, which is~$\overline{\FZ}_k^2/\overline{\FZ}_k^1$ by~(\ref{eq12}).
\end{proof}

The result from~\cite{GKZ} cited in theorem~\ref{thmgkz} proves more than the existence of~a space of linear relations between single and double zetas arising of dimension equal to that of~$E_k$, however; it also states that the coefficients of these relations are essentially the coefficients of the period polynomials under the change of variables~$X\leftarrow X+Y$. We now show that also this result can be deduced from studying the matrix~$A$, which has some particular symmetry properties proven in the following propositions.

\begin{prop}\label{propsym}
Let~$k\ge 12$ be even, and define~$\frac{k-4}{2}\times \frac{k-4}{2}$ matrices~$D$ and~$B$ by
\[D^{-1}=\diag\left(\binom{k-2}{2i}\right)\text{,}\qquad B_{ij}=\binom{2j}{2i}\text{.}\]
Then~$\transp ADB$ is symmetric.
\end{prop}

\begin{proof}
The~$(i,j)$-th entry of~$\transp ADB$ is given by
\begin{align*}
  &  \sum_{r=1}^{\frac{k-4}{2}} \left(\binom{2i}{2r}-\binom{2i}{k-2-2r}+\delta_{i+r}^{\frac{k-2}{2}}\right)\cdot
\binom{k-2}{2r}^{-1}\binom{2j}{2r} \\
 &  =\sum_{r=1}^{\frac{k-4}{2}} \frac{\binom{2i}{2r}\binom{2j}{2r}}{\binom{k-2}{2r}} - \sum_{r=1}^{\frac{k-4}{2}} \frac{\binom{2i}{k-2-2r}\binom{2j}{2r}}{\binom{k-2}{2r}} + \frac{\binom{2j}{k-2-2i}}{\binom{k-2}{k-2-2i}}.
\end{align*}
The left-hand term is obviously symmetric in~$i$ and~$j$, and so is the middle term, using the index change~$r\leftarrow \frac{k-2}{2}-r$. As for the last term, it is equal to 
\[\frac{(2i)!\,(2j)!}{(k-2)!\,(2i+2j-k+2)!}\text{,}\]
so it is also symmetric, which concludes the proof.
\end{proof}

\begin{ex}
Let~$A$ be the matrix in weight~$12$ given explicitly above. Then 
\[\transp ADB={}^tBDA=\frac{1}{630}
\begin{pmatrix}
 14 & 84 & 210 & 392\\
84 & 507 & 1305 & 2478\\
210 & 1305 & 3783 & 7644\\
392 & 2478 & 7644 & 15890
\end{pmatrix}
\cdot\]
The kernel of this matrix (and its transpose) is of course still generated by the same vector~$\transp(1,-3,3,-1)$ as the kernel of~$A$.
\end{ex}

\begin{prop}
Let~$k\ge 12$ be even and suppose that~$\transp (a_1,a_2,\ldots,-a_2,-a_1)\in \Ker A$. Set 
\[P(X,Y)=\sum_{i=1}^{[\frac{k-4}{4}]} a_i\,(X^{2i}Y^{k-2-2i}-X^{k-2-2i}Y^{2i})\]
and define the coefficients~$q_{r,k-r}$ for~$1\le r\le k-3$ by 
\[P(X+Y,Y)=\sum_{r=1}^{k-3} \binom{k-2}{r-1}\, q_{r,k-r}\,
X^{r-1}Y^{k-r-1}\text{.}\]
Then the vector~$\transp (q_{3,k-3},\ldots,q_{k-3,3})$ (with odd indices) lies in the kernel of~$\transp A$, and in fact the kernel of~$\transp A$ consists in exactly these vectors.
\end{prop}

\begin{proof}
We first compute the~$q_{r,k-r}$ in terms of the~$a_i$. We have
\begin{align*}
P(X+Y,Y) & =\sum_{i=1}^{[\frac{k-4}{4}]} a_i\,
\left((X+Y)^{2i}Y^{k-2-2i}-(X+Y)^{k-2-2i}Y^{2i}\right) \\ 
 & =\sum_{i=1}^{[\frac{k-4}{4}]} a_i\,\left(
\sum_{r=1}^{2i} \binom{2i}{r}\,X^rY^{k-2-r}
-\sum_{r=1}^{k-2-2i} \binom{k-2-2i}{r}\,X^rY^{k-2-r}\right) \\ 
 & =\sum_{i=1}^{[\frac{k-4}{4}]} a_i\left(
\sum_{r=1}^{k-2} \binom{2i}{r}\,X^rY^{k-r-2}
-\sum_{r=1}^{k-2} \binom{k-2-2i}{r}\,X^rY^{k-2-r}\right)
\end{align*}
since we can freely replace the upper limits on the~$r$-sums by larger ones, as the binomial coefficients will simply be equal to zero. Thus we can invert the order of the sums and write
\[P(X+Y,Y)=
\sum_{r=1}^{k-2}\sum_{i=1}^{[\frac{k-4}{4}]} a_i\left(
\binom{2i}{r}-\binom{k-2-2i}{r}\right)\,
X^rY^{k-r-2}\text{.}\]

We are interested in the coefficients~$\binom{k-2}{r-1}\,q_{r,k-r}$ of the monomials~$X^{r-1}Y^{k-r-1}$ where~$r-1$ is even, so we write~$r-1=2j$, and the coefficient of~$X^{2j}Y^{k-2-2j}$ is then given by
\[\binom{k-2}{2j}\,q_{2j+1,k-2j-2}=
\sum_{i=1}^{[\frac{k-4}{4}]}a_i\left(\binom{2i}{2j}
-\binom{k-2-2i}{2j}\right)\text{,}\qquad 1\le j\le \frac{k-4}{2}\text{.}\]

Now, note that since~$\transp ADB$ is symmetric by proposition~\ref{propsym}, we have~$\Ker\transp ADB=\Ker {}^tBDA=\Ker A$. Thus since~$D$ and~$B$ are both invertible, we have~$\Ker\transp A=DB\Ker A$. But by the definitions of~$D$ and~$B$, the~$j$-th component of the vector~$DB\transp(a_1,\ldots,-a_1)$ is exactly equal to~$q_{2j+1,k-2j-1}$ (indeed, the whole point of the matrix~$DB$ is to effect the variable change from the~$a_i$ to the~$q_{r,k-r}$). This concludes the proof.
\end{proof}

Since we saw above that the vectors in the kernel of~$\transp A$ provide the coefficients of the linear combinations of odd-component double zetas which are zero modulo~$Z(k)$, we recover the statement of theorem~\ref{thmgkz} (the~\cite{GKZ} result) in the formal zeta algebra~$\overline{\FZ}$, and thus also for real multizeta values by passage to the quotient, as an immediate corollary of these propositions.

\begin{coro}[\cite{GKZ} theorem]
Let~$k\ge 12$ be an even integer, let~$P(X,Y)=\sum_{i=1}^{[\frac{k-4}{4}]} a_i\,(X^{2i}Y^{k-2-2i}-X^{k-2-2i}Y^{2i})
\in E_k$ be~a homogeneous period polynomial of weight~$k$, and write
\[P(X+Y,Y)=\sum_{r=1}^{k-3} \binom{k-2}{r-1}\, q_{r,k-r}\,
X^{r-1}Y^{k-r-1}.\]
Then the linear combination
\[\sum_{\substack{r=3 \\ r\text{ odd}}}^{k-3} q_{r,k-r}\,Z(r,k-r)\]
is equal to~a scalar multiple of~$Z(k)$ in~$\overline{\FZ}$.
\end{coro}

\begin{rem}
 A variant of the matrix~$A$ can also be found in~\cite{Kan}, but the equivalence we proved above was not derived from it.
\end{rem}

\newpage
\vspace*{\stretch{1}}
\bibliographystyle{amsalpha}
\bibliography{bibdz}
\vspace*{\stretch{2}}

\end{document}